\documentclass[a4paper,11pt]{article}
\usepackage{geometry}
\usepackage{amsmath}
\usepackage{amsfonts}
\usepackage{amssymb}
\usepackage{color}

\newtheorem{theorem}{Theorem}

\newtheorem{definition}[theorem]{Definition}

\newtheorem{lemma}[theorem]{Lemma}

\newtheorem{remark}[theorem]{Remark}

\newtheorem{hypothesis}[theorem]{Hypothesis}
\newenvironment{proof}[1][Proof]{\textbf{#1.} }{\ \rule{0.5em}{0.5em}}
\def\RR{\mathbb{R}}
\def\E{\mathbb{E}}

\begin{document}

\title{ Flow of diffeomorphisms for SDEs
with unbounded
H\"older continuous drift
}
\author{F. Flandoli$^{1}$,
M. Gubinelli$^{2}$, E. Priola$^{3}$\\{\small {(1) Dipartimento di
Matematica Applicata ``U. Dini'', Universit\`a di Pisa,
Italia}}\\{\small { (2) CEREMADE  (UMR 7534), Universit\'e Paris
Dauphine, France}}\\{\small { (3) Dipartimento di Matematica,
Universit\`a di Torino, Italia }}}\maketitle
\begin{abstract}
We consider a SDE with a smooth multiplicative non-degenerate noise
and a possibly unbounded H\"{o}lder continuous drift term. We prove
existence of a global flow of diffeomorphisms by means of a special
transformation of the drift of It\^{o}-Tanaka type. The proof
requires non-standard elliptic estimates in H\"{o}lder spaces.
 As an application of the
stochastic flow, we  obtain
a Bismut-Elworthy-Li type formula for the first derivatives of
 the associated  diffusion semigroup.
\end{abstract}

\section{Introduction}

In this paper we study the existence of a global stochastic flow of
diffeomorphisms for the following stochastic differential equation in
$\mathbb{R}^{d}$
\begin{equation}
dX_{t}^{x}=b\left(  X_{t}^{x}\right)
dt+\sum_{i=1}^{k}\sigma_{i}\left( X_{t}^{x}\right)  dW_{t}^{i},\quad
t\geq0,\quad X_{0}^{x}=x, \label{SDE}
\end{equation}
where $W_{t}=(W_{t}^{1},...,W_{t}^{k})$ is a standard Brownian
motion in ${\mathbb{R}}^{k}$. We assume that the diffusion
coefficients
$\sigma_i : \RR^d \to \RR^d$, $ i=1,\dots,k$,   are smooth and
 non-degenerate and we allow
 the drift term
$b: \RR^d \to \RR^d $ to  be unbounded and H\"older continuous.

Following a common language, we say that equation (\ref{SDE}) is
\textit{weakly complete} if there exists a unique global strong solution for
every $x\in\mathbb{R}^{d}$, and that it is \textit{strongly complete} if there
exists a global stochastic flow of homeomorphisms.
If the coefficients $b$ and
$\sigma_{i}$ are globally Lipschitz, then one has strong completeness (see
\cite{K} and \cite{K1}).

Weak completeness is true under much weaker assumptions: for
instance, when the coefficients $b$ and $\sigma_{i}$ are locally
Lipschitz continuous and have at most linear growth.  In dimension one,
these assumptions also imply strong completeness (see \cite{K}  and
\cite{K1})  but in dimension larger than one there are
counterexamples, from \cite{LS}, even in the case of smooth bounded
coefficients.  These examples indicate that some form of global
control at infinity on the increments of the coefficients is
necessary. For (at least) locally Lipschitz coefficients, there are
indeed positive results of strong completeness (see \cite{FIZ},
\cite{L}, \cite{MS}).

Strong completeness for non-locally Lipschitz coefficients can be
established replacing the global Lipschitz condition on the
coefficients with  global log-Lipschitz type conditions (see
\cite{RZ}, \cite{Za1}, \cite{FZ}, \cite{FL}).  Such log-Lipschitz
conditions are stronger than the  H\"older continuity.

Many papers prove weak completeness for SDEs with non-locally
Lipschitz continuous coefficients assuming a non-degenerate
diffusion matrix $\sigma$. First papers in this direction were
\cite{Zv} and \cite{V} in which the method of  the so called
Zvonkin's transformation was introduced. More recent papers dealing
with such approach are \cite{GM}, \cite{Kry-Ro}, \cite{Za},
\cite{Za2}  (see also the references therein).
%
In the case of non-degenerate additive noise and time dependent drift $b$, the
most advanced result (but see
%
also the 1-dimensional results reported in \cite{RY}) is \cite{Kry-Ro}; in
such paper it is shown that it is sufficient to assume that $b\in L^{q}\left(
0,T;L^{p}_{loc}\left(  \mathbb{R}^{d}\right)  \right)  $ with $\frac{d}%
{p}+\frac{2}{q}<1$, $p \ge2$ and $q>2$, plus a non-explosion condition, to get
weak completeness. This result has been generalized in \cite{Za}  to cover
also the case in which $\sigma$ is variable,  time-dependent and
non-degenerate. We do not know about strong completeness under such weak assumptions.

%

The contribution of the present paper is to prove strong
completeness for SDEs with ``locally uniformly $\theta$-H\"{o}lder
continuous'' drift $b$, for some $\theta\in\left(  0,1\right)  $
(see \eqref{hol}), removing boundedness of $b$ or additional
regularity assumed in previous works. Also, we allow non-degenerate,
bounded
and  $C^{3}_b (\RR^d, \RR^d) $-diffusion coefficients
$(\sigma_i)_{i=1,\dots,k}$. We point out that
 our result seems to be new
even in the case of constant and non-degenerate
$(\sigma_i)_{i=1,\dots,k}$.

In spite of the fact that $b$ is not even differentiable, under the
previous assumptions, we construct a stochastic flow of
$C^{1}$-diffeomorphisms (see Theorem \ref{th:flow1}) using the
approach of \cite{FGP} rather than the Zvonkin's transformation
method used in the above mentioned works on strong completeness (we
compare the two methods in Section 3).

In \cite{FGP} in order to study a linear stochastic transport
equation with a  {\it bounded vector field} $\tilde b (t,x)$ which
is H\"older continuous  in $x$, uniformly in time,  we have showed
that if in \eqref{SDE} $\sigma= (\sigma_i)$ is constant and
non-degenerate and $b  = \tilde b$,  then there exists a stochastic
flow of $C^{1}$-diffeomorphisms. This  result can be extended
without difficulties to the case in which $\sigma$  is not constant,
bounded, non-degenerate,  and time-dependent  (see \cite{Za2} where
this case  is investigated  by the Zvonkin's transformation or
Remark \ref{vi} where we show such result following the approach of
\cite{FGP}).

In the present situation, since our $b$ is {\it unbounded,}
we need new
global regularity results in H\"older spaces  for
the solution $u$ of the  elliptic equation
\begin{equation}
\lambda u (x) - \frac{1}{2} \mathrm{Tr} (a(x) D^{2}u (x)) - b(x)\cdot D u(x)=
b(x),\;\;\; x \in\mathbb{R}^{d},
\end{equation}
to be interpreted componentwise, where $\lambda>0$ is large enough,
$a(x) = \sigma(x) \sigma^{*} (x)$ ($\sigma^{*}(x)$ denotes the
adjoint matrix of $\sigma(x)$).  The study of this equation will be
the subject of Section 2 of the present paper. The required
estimates are  not covered  by recent papers dealing with elliptic
and parabolic equations with  unbounded coefficients  (compare with
\cite{Ce}, \cite{BL}, \cite{KP} and the references therein). To
obtain such result we prove a crucial Lemma \ref{semi}
concerning
estimates on the derivatives of the associated diffusion semigroup
  when it is applied to \textit{unbounded}
functions $f$; in its proof we also use an argument from the proof
of \cite[Theorem 3.3]{Pr2}.  In Remark \ref{fine} we
 show a possible extension of our Theorem \ref{th:flow1}
 to the case in which
$b$ and $\sigma$ are  time-dependent.

We finish the paper by showing that a Bismut-Elworthy-Li formula
holds for the diffusion semigroup associated to \eqref{SDE}
 (see
Theorem \ref{bismut}). Under
the poor regularity of $b$ assumed here, this result is new.
Bismut-Elworthy-Li formula requires a suitable form of
differentiability of the solution of \eqref{SDE} with respect to the
initial condition $x$;\ we have this result as a byproduct of our
Theorem \ref{th:flow1} on existence of a differentiable stochastic
flow.

\paragraph{Notations and assumption}

%
The euclidean norm in any $\mathbb{R}^{k}$, $k\geq1$, will be denoted by
$|\cdot|$ and its inner product by $\cdot$ or $\langle\cdot,\cdot\rangle$.
For $\theta\in(0,1)$, we define the  set $C_{}^{\theta}(\mathbb{R}%
^{d};\mathbb{R}^{k})$, $k,\,d\geq1$, as set of all vector-fields
$f:\mathbb{R}^{d}\rightarrow\mathbb{R}^{k}$ for which%
\begin{equation} \label{hol}
\lbrack f]_{\theta}:=\sup_{x\neq y\in\mathbb{R}^{d},\left|
x-y\right|  \leq 1}\frac{|f(x)-f(y)|}{|x-y|^{\theta}}<\infty.
\end{equation}
These are the {\it ``locally uniformly $\theta$-H\"{o}lder
continuous''} vector fields mentioned in the introduction. The
function $f\left( x\right) =|x|^{\theta}$ is a classical example. We
let
\begin{equation} \label{vai}
[ f]_{\theta,1}:=\sup_{x\neq
y\in\mathbb{R}^{d}}\frac{|f(x)-f(y)|}{(|x-y|^{\theta
}\vee|x-y|)}<\infty,
\end{equation}
where $a\vee b=\max(a,b)$, for $a,b\in\mathbb{R}$. By a simple
argument we have $ [f]_\theta \le [f]_{\theta,1} \le 2 [f]_{\theta}
$, so in particular  functions in $ C_{}^{\theta}(\mathbb{R}%
^{d};\mathbb{R}^{k})$ have {\it at most linear growth.} The set
$C_{}^{\theta}(\mathbb{R}^{d};\mathbb{R}^{k})$ becomes a Banach
space with respect to the norm
\[
\Vert f\Vert_{\theta}=\big\|\,{(1+|\cdot|)^{-1}}\,{f(\cdot)}\big\|%
_{0}+[f]_{\theta},
\]
where $\Vert\cdot\Vert_{0}$ denotes the supremum norm over
$\mathbb{R}^{d}$. We say that $f\in
C_{}^{n+\theta}(\mathbb{R}^{d};\mathbb{R}^{k})$, $n \ge 1$,
 if $f\in
C_{}^{\theta}(\mathbb{R}^{d};\mathbb{R}^{k})$ and moreover, for all
$i=1,\dots,n,$ the Fr\'{e}chet derivatives $D^{i}f$ are
\textit{bounded} and $\theta$-H\"{o}lder continuous. Define the
corresponding norm as
\begin{equation}
\Vert f\Vert_{n+\theta}=\Vert f\Vert_{\theta }+\sum_{i=1}^{n}\Vert
D^{i}f\Vert_{0}+[D^{n}f]_{\theta}.
\end{equation}
If $\mathbb{R}^{k}=\mathbb{R}$, we simply write
$C_{}^{n+\theta}(\mathbb{R}^{d})$ instead of
$C_{}^{n+\theta}(\mathbb{R}^{d};\mathbb{R})$, $n\geq0$.
$C_{b}^{n+\theta}(\mathbb{R}^{d};\mathbb{R}^{k})$ is the subspace of
$C_{}^{n+\theta}(\mathbb{R}^{d};\mathbb{R}^{k})$, consisting of all
bounded functions of
$C_{}^{n+\theta}(\mathbb{R}^{d};\mathbb{R}^{k})$.
In particular, $C^{\theta}_b(\mathbb{R}^{d})$ is the usual
Banach space of all real bounded and $\theta$-H\"older continuous functions
on $\RR^d$ (cf. \cite{Kr}).  $C_{b}^{n}(\mathbb{R}^{d};\mathbb{R}^{k})$
is the space of all
bounded functions from $\RR^d$ into $\RR^k$ having also bounded
derivatives up to the order $n \ge 1$ and we set
$C_{b}^{n}(\mathbb{R}^{d};\mathbb{R}) = C_{b}^{n}(\mathbb{R}^{d})$.
Finally,  we say that $f:\mathbb{R} ^{d}\rightarrow
\mathbb{R}^{d}$ is of class
$C^{n,\alpha }$, $n \ge 1 $, $\alpha \in (0,1)$, if $f$ is
continuous on $\mathbb{R}^{d}$, $n$-times differentiable
and the derivatives up to the order $
n$ are  $\alpha$-H\"{o}lder continuous on each compact set of
$\RR^d$.

\vskip1mm Throughout the paper we will assume a fixed stochastic basis with a
$d$-dimensional Brownian motion $\left(  \Omega,\left(  \mathcal{F}{}%
_{t}\right)  ,{}\mathcal{F},P,\left(  W_{t}\right)  \right)  $ to be given.
Denote by $\mathcal{F}_{s,t}$ the  completed $\sigma$-algebra generated by
$W_{u}-W_{r}$, $s\leq r\leq u\leq t$,  for each $0\le s<t$.

On equation \eqref{SDE}, we will consider the following assumptions.

\begin{hypothesis}
\label{hy1} There exists ${\theta}\in(0,1)$ such that $b\in
C_{}^{\theta }(\mathbb{R}^{d};\mathbb{R}^{d})$.

\end{hypothesis}
\begin{hypothesis}
\label{hy2}  The diffusions coefficients
 $\sigma_{i}: \mathbb{R}^d \to \mathbb{R}^d$, $i=1, \ldots, k$,
    are bounded functions of class
$C^{3 }_b (\mathbb{R}^d, \mathbb{R}^d).$
\end{hypothesis}

\begin{hypothesis}
\label{hy3} Consider the $d \times k$ matrix $\sigma(x) =
(\sigma_i(x))$, and
its adjoint matrix $\sigma^*(x)$,
 $x \in \mathbb{R}^d$;  we assume that,
for any $x \in \mathbb{R}^d$, there exists the inverse
 of $a(x)= \sigma (x)\sigma^* (x)$ and
\begin{equation}\label{si}\| a^{-1}\|_0  =  \sup_{x \in \mathbb{R}^d}
\| a^{-1} (x) \| < \infty
\end{equation}
\end{hypothesis}
($\| a^{-1} (x) \|$ denotes the Hilbert-Schmidt norm of
 the $d \times d$ symmetric matrix $a^{-1} (x)).$

\section{Regularity results for   the associated elliptic problem
}

\subsection{Estimates on the derivatives of the diffusion semigroup}

Here, we consider the SDE \eqref{SDE}, assuming that
$\sigma$ satisfies Hypotheses \ref{hy2} and \ref{hy3} and
 imposing
in addition that
\begin{equation} \label{hyy}
b \in C^3(\mathbb{R}^d;\mathbb{R}^d) \;\;
\text{with all bounded derivatives
up to the third order.}
\end{equation}
Clearly this is  stronger
 than   Hypothesis   \ref{hy1} but  $b$ is
  not assumed to be bounded.

\smallskip Let $(P_t)$ be the corresponding diffusion semigroup,
i.e., for any $g: \mathbb{R}^d \to \mathbb{R}$ Borel and bounded,
$$
P_t g (x) = \mathbb{E} [g(X_t^x)],\;\;\; x \in\mathbb{R}^d,\;\; t
\ge 0,
$$
where $(X_t^x)$ is the unique strong solution to \eqref{SDE}
 under \eqref{hyy}.

In our next result,
 we will prove  estimates on the spatial derivatives of
$P_t f$, $t>0,$ assuming that $f \in C^{\theta} (\mathbb{R}^d)$.
 To this purpose,
 we will use  the so-called Bismut-Elworthy-Li
 formula  (see \eqref{bism})
  for the spatial derivatives of $P_t f$ (cf. \cite{EL}).

Let us comment on  such  formula. Probabilistic formulae for
the spatial derivatives of Markov semigroups    have
been much studied
  for different classes of
  degenerate and non-degenerate
 diffusion processes  even with jumps (see \cite{B}, \cite{KS},
 \cite{EL}, \cite{DZ1}, \cite{Ce}, \cite{Fu} \cite{Pr1}, \cite{Pr2},
 \cite{Za2}
 and the references therein). The martingale
  approach of \cite{EL} mainly
 works for
  non-degenerate semigroups (but see also \cite[Chapter 3]{Ce}
  and  \cite{Za2}); it    has been also used for
  some infinite dimensional diffusion processes (see  \cite{DZ1}
 and   \cite{Ce}).  On the other hand, in case of degenerate
 diffusion semigroups,  more complicate formulae
  for the  derivatives
   can be established by
  Malliavin Calculus  (see \cite{B},
  \cite{KS}, \cite{Fu} and
  \cite{Pr1}).
 Some applications to Mathematical Finance  are given in  \cite{F}.

\vskip 2mm  The next lemma is of independent interest
since the function
$f$ in \eqref{gra} is not assumed to be bounded
(compare with \cite[Chapter
1]{Ce} and \cite[Chapter 6]{BL}).

\begin{lemma} \label{semi}
Assume Hypotheses \ref{hy2} and
 \ref{hy3} and condition  \eqref{hyy}.
  There exist constants $c_j>0$, $M_j>0$,
$j=1, 2,3$ ($c_j$ and $M_j$ depends
 on $\theta$, $\| a^{-1}\|_0$, $d,$ $\| \sigma\|_0$ and on the supremum
 norms of derivatives of $\sigma$ and $b$ up to the order $j$),
     such that,  for any $f \in C^{\theta} (\mathbb{R}^d)$, $t>0$,
it holds
\begin{equation} \label{gra}
\| D^j P_t f \|_0 \le    M_j  [f]_{\theta }  \,
\frac{e^{c_j  \,  t}}{t^{(j -\theta)/2}},\;\;\; t>0, \;\; \text{for}
\;\; j=1, 2,3.
\end{equation}
\end{lemma}
\begin{proof} {\it I Step.}
 First note  that $\mathbb{E}
 [\sup_{t \in [0,T]}
|X_t^x|^{q}]  \le  C_T(1 + |x|^q)$, for any $T>0$, $x \in
\mathbb{R}^d$, $q \ge 1$ (see, for instance, \cite[Chapter II]{K}).

\smallskip  It is also known  that, for any $t \ge 0$, the
mapping:
\begin{align}\label{frec}
x \mapsto
X_t^x \;\;
\text{is three times Fr\'echet differentiable from $\RR^d$ into
$L^2 (\Omega)$}
\end{align}
(see \cite[Section 1.3]{Ce} which contains
a more general result). Let us
write  the Fr\'echet derivatives:
$$
\eta_t (x, h) = D_x (X_t^x) [h], \;\; \xi_t (x,h,k) =
D_x^2 (X_t^x) [h, k],\;\; \psi_t (x,h,k,l) =
 D_x^3 (X_t^x) [h, k,l],
$$
for any $x,h,k,l \in \RR^d$. These derivatives satisfy suitable
stochastic variation equations (see  \cite[Chapter II]{K}). We only
write down the variation equation  for $\eta_t = \eta_t (x,h)$:
$$
d \eta_t = Db (X_t^x) \eta_t + D \sigma (X_t^x) \eta_t
dW_t,\;\;\; \eta_0 =h.
$$
Using standard estimates, based on the Burkholder inequality, we get
that, for any $p \ge 1$, that there exist positive constants $C$ and
$c$ (depending on $p$,  $\| Db\|_0$ and  $\| D \sigma \|_0$) such
that, for any $x \in \RR^d$, $h \in \RR^d$,
\begin{equation} \label{lip}
\E  |\eta_t (x,h)|^p  \le C |h|^p e^{c t},\;\; t \ge 0.
\end{equation}
In a similar way, using the second and third variation equations, we
obtain the estimates:
\begin{equation} \label{lip1}
\E  |\xi_t (x,h,k)|^p  \le C_2 |h|^p |k|^p e^{\hat c_2 t},\;\;
\end{equation}
$$
\E  |\psi_t (x,h,k,l)|^p  \le C_3 |h|^p |k|^p |l|^p \,  e^{\hat
c_3
t},\;\;
t \ge 0,
$$
for any $x, h,k,l \in \RR^d$ (with   positive constants $C_i $ and
$\hat c_i$ which depend on $p$ and on the supremum norms of the
derivatives of $b$ and $\sigma$ up to the  order $i$, $i=2,3$).

\smallskip {\it II Step.}
Arguing similarly to   \cite[Section 1.5]{Ce} one can prove that,
for any $f \in C^{\theta } (\mathbb{R}^d) $, $t>0$,  the map: $x
\mapsto P_t f (x)$ is differentiable  on $\mathbb{R}^d$
and, moreover, we have the following Bismut-Elworthy-Li formula:
\begin{equation} \label{bism}
\langle D P_t f(x), h \rangle =
\mathbb{E} \Big [ f(X_t^x) \, J^1 (t,x,h) \Big], \;\;
x,\, h \in \mathbb{R}^d, \; t>0,
 \; \mbox{where}
\end{equation}
$$
J^1 (t,x,h)= \frac{1}{t}
 \int_0^t \langle \sigma^{*} (X_s^x) \,  a^{-1}(X_s^x)
\, \eta_s (x,h) , dW_s
  \rangle.
$$
Note that   formula \eqref{bism} is first proved
for bounded $f\in C^2_b (\mathbb{R}^d)$.
Then a
straightforward approximation argument shows that
\eqref{bism} holds even for (a possibly unbounded)
 $f \in C^{\theta}
 (\mathbb{R}^d)$.
However, to be precise, in \cite{Ce}, it is assumed that
$\sigma(x)$ is an invertible $d \times d $ matrix and so
 the expression of  $J^1$ in \cite[Section 1.5]{Ce} contains
$\sigma^{-1} (X_s^x)$ instead of our $\sigma^{*} (X_s^x) \,
a^{-1}(X_s^x)$. We briefly explain why \eqref{bism} holds
 following the proof of
\cite[Theorem 5.1]{PZ}. We   only discuss the crucial point
of the argument
which is needed to get \eqref{bism}
when $f\in C^2_b (\mathbb{R}^d)$. One has  by the It\^o formula
$$
f(X_t^x) = P_t f(x) +
 \int_0^t \langle  DP_{t-s}f(X_s^x) ,
 \sigma (X_s^x) dW_s \rangle.
$$
Multiplying both terms of the identity by the martingale
 $$K_t  =
\int_0^t \langle \sigma^{*} (X_s^x) \,  a^{-1}(X_s^x)
\, \eta_s (x,h) , dW_s \rangle,
$$
and taking the expectation, one
arrives at
$$
\mathbb{E} [f(X_t^x) K_t] = \int_0^t \mathbb{E} [\langle DP_{t-s}f(X_s^x)
, \eta_s (x,h) \rangle ] ds =
 t \langle DP_t f(x), h\rangle.
$$
Thus   \eqref{bism} is proved.

\smallskip Now {\it the  problem is to show that, for  $ f \in
C^{\theta}
 (\mathbb{R}^d)$, $t>0$, the map:
   $x \mapsto \langle D P_t f(x), h \rangle$ is a bounded
 function (we cannot use as in \cite{Ce} the boundedness of
   $f$).}

\vskip 1mm  By using \eqref{lip}, we get easily that  there exist
$C_1 >0$ depending on  $\| a^{-1}\|_0$, $ \|Db \|_0$
 and $\| D \sigma\|_0$ such that
\begin{equation} \label{f4}
\mathbb{E} |J^1 (t,x,h)|^2 \le
\frac{C_1 e^{C_1 t}}{t} |h|^2,\;\;\; t > 0.
\end{equation}
Now  we prove the  crucial estimate of the first derivative in
\eqref{gra}.  We use an argument from the proof of \cite[Theorem
3.3]{Pr2}.
Introduce the deterministic process
$$
Y_t^x =  x + \int_0^t b(Y_s^x)ds,\;\; t \ge 0, \; x \in
\mathbb{R}^d,
$$   which solves $ \dot  {Y^x_t }  = b(Y^x_t), \;\;\;
     Y^x_0  = x. $
Using that $\sigma$ is bounded and applying the Gronwall lemma, we
find, for any $q \ge 1$,
\begin{equation} \label{f9}
\mathbb{E} |X_t^x - Y_t^x |^q \le M t^{q/2} \,
 e^{c_1 t },\;\;\; t \ge 0, \; x \in
\mathbb{R}^d,
\end{equation} where $M$ depends on $\| \sigma\|_0$ and $q$
 and $c_1 $ on $\|
Db\|_0$ and $q$. Since
     $$
\mathbb{E} \big [ f(Y_t^x) \, J^1 (t,x,h) ]
= f(Y_t^x) \langle D (P_t 1)(x), h \rangle
=0,\;\; t>0,\;\; h \in \mathbb{R}^d,
\; x \in \mathbb{R}^d,
$$
we have (see also \eqref{vai})
$$
| \langle D P_t f(x), h \rangle|  = \Big|
 \mathbb{E} \big [ (f(X_t^x)- f(Y_t^x)) \, J^1 (t,x,h) \big] \Big|
$$
$$
\le  2[f]_{\theta} \, \mathbb{E} \big[ (|X_t^x - Y_t^x |^{\theta} \,
 \vee |X_t^x - Y_t^x |^{}) \, \,
|J^1 (t,x,h)| \big]
$$ \begin{equation}
\label{dfr}   \le 2[f]_{\theta} \, \big(\mathbb{E} \big[
|X_t^x- Y_t^x |^{2\theta} \vee |X_t^x - Y_t^x |^{2} \big]
\big)^{1/2} \,
    (\mathbb{E} |J^1 (t,x,h)|^2)^{1/2},
\end{equation}
$t>0$. Using that $a \vee b \le a + b$, $a, b \ge 0$,
 and the previous estimates \eqref{lip} and
 \eqref{f9}, we find
\begin{align} \label{s7}
| \langle D P_t f(x), h \rangle| \le C''
[f]_{\theta}(t^{\theta/2} + t^{1/2})
\frac{e^{c''   t}}{ t^{1/2}}|h| \le
[f]_{\theta}  \frac{C' e^{c'   t}}{ t^{1/2 - \theta/2}} |h|, \;\;
t>0,\; x \in \mathbb{R}^d,
\end{align}
where $C'$ and $c$ depend on $ \| \sigma \|_0,$ $\|
a^{-1} \|_0$, $\| D \sigma \|_0$,
$\| D b \|_0$ and $\theta$.

\medskip Let us consider the remaining estimates in \eqref{gra}.
We have, using the semigroup law,  $P_t f = P_{t/2}
 (P_{t/2} f)$ and so  (cf. \cite[formula (1.5.2)]{Ce}),
  for any
   $x,\, h, \, k \in \RR^d,$ $  t>0,$
$$
\langle D^2 (P_t f) (x) k, h \rangle
=
D_k \Big( \mathbb{E} \Big [ (P_{t/2}f)(X_{t/2}^{(\cdot)})
\, J^1 (t/2,(\cdot),h) \Big] \Big)(x)
$$
$$
= \mathbb{E} \Big [ \langle D P_{t/2}f(X_{t/2}^{x}), \eta_{t/2}
(x, k) \rangle
\, J^1 (t/2,x,h) \Big] \, +
\, \mathbb{E} \Big [  P_{t/2}f(X_{t/2}^{x})
\, D_k J^1 (t/2,x,h) \Big] $$$$
= \Gamma_1(t,x) + \Gamma_2(t,x),
$$
where $D_k$ denotes the directional derivative along the vector $k$
(indeed, for any fixed $t>0$ and $h \in \RR^d$,
the mapping:  $x \mapsto J^1 (t/2,x,h)$ is
  Fr\'echet differentiable from $\RR^d$ into
$L^2 (\Omega)$; this follows easily, using \eqref{frec},
 \eqref{f9}, \eqref{lip} and \eqref{lip1}).
We have
$$ D_k J^1 (t/2,x,h)= \frac{2}{t}
 \int_0^{t/2} \langle D\sigma^{*} (X_s^x)[\eta_s (x,k) ]
  \,  a^{-1}(X_s^x)
\, \eta_s (x,h) , dW_s
  \rangle
$$
$$
- \frac{2}{t}
 \int_0^{t/2} \langle \sigma^{*} (X_s^x)
  \,   a^{-1}(X_s^x) \, Da (X_s^x) [\eta_s (x,k) ] \,
  a^{-1}(X_s^x)
\, \eta_s (x,h) , dW_s
  \rangle
  $$$$
+\frac{2}{t}
 \int_0^{t/2} \langle \sigma^{*} (X_s^x)
  \,   a^{-1}(X_s^x)
\, \xi_s (x,h,k) , dW_s
  \rangle.
$$
 Using the Schwarz inequality, \eqref{f4} and
 $$ \sup_{x \in \mathbb{R}^d} (\E| \langle
 D P_{t/2} f(X^x_{t/2}), \eta_{t/2} (x,k)
 \rangle |^2)^{1/2}
 \le [f]_{\theta}  \frac{C'' e^{c'  \theta t}}{ t^{1/2 - \theta/2}}
|k|,$$
we get immediately $ |\Gamma_1(t,x) | \le    M  [f]_{\theta }  \,
\frac{e^{c  \,  t}}{t^{(2 -\theta)/2}}|h| |k|$, $t>0$,
$x \in \RR^d$. To estimate
$\Gamma_2$, first note that, by taking $f=1$,
$$
0=  \langle D^2 (P_t 1) (x) k, h \rangle = 0 +
 \mathbb{E} \big [
\, D_k J^1 (t/2,x,h) \big],
$$
for any $x, h, k \in \RR^d$. We find (arguing similarly to
\eqref{dfr})
$$ \Gamma_2(t,x)=
\mathbb{E} \Big [  \big(P_{t/2}f(X_{t/2}^{x}) -
P_{t/2}f(Y_{t/2}^{x}) \big)
\, D_k J^1 (t/2,x,h) \Big].
$$
Since
$$
|P_{s}f(x) - P_{s}f(y)| \le \E |f (X_{s}^{x}) -
f(X_{s}^{y})| \le 2 [f]_{\theta}
 \mathbb{E} \big[ (|X_s^x - X_s^y |^{\theta} \,
 + |X_s^x - X_s^y |^{})
$$
$$ \le  2 [f]_{\theta} M (|x-y|^{\theta} + |x-y|)
 e^{c_1 s },\;\;\; s \ge 0, \;\; x, \, y \in
\mathbb{R}^d,
$$
we find, for any $x \in \RR^d$, $t>0,$
$$
|\Gamma_2(t,x)| \le 2  M
 e^{c_1 t/2 } [f]_{\theta} \, \mathbb{E} \big[ (|X_{t/2}^x - Y_{t/2}^x
|^{\theta} \,
 + |X_{t/2}^x - Y_{t/2}^x |^{}) \, \,
|D_k J^1 (t/2,x,h)| \big].
$$
$$
\le 2  M
 e^{c_1 t/2 } [f]_{\theta}  \,
 \big(\mathbb{E} \big[
|X_{t/2}^x- Y_{t/2}^x |^{2\theta} + |X_{t/2}^x - Y_{t/2}^x |^{2}
\big] \big)^{1/2} \,
    (\mathbb{E} |D_k J^1 (t/2,x,h)|^2)^{1/2}
$$
$$
\le [f]_{\theta}  \frac{C_1 e^{c_1   t}}{ t^{1/2 - \theta/2}}
|k||h|,
$$
where $C_1$ and $c_1$ depends on
 $ \| \sigma \|_0,$ $\|
a^{-1} \|_0$, $\| D \sigma \|_0$, $\| D^2 \sigma \|_0$
$\| D b \|_0$, $\| D^2 b \|_0$ and $\theta$.
We have so obtained estimate in \eqref{gra} corresponding
to $j =2$.

The estimate for $j =3$ follows in a similar way.
\end{proof}

\subsection{ The main regularity result  }

With respect to the previous section, here   we   consider the elliptic operator
$$
L u(x)  = \frac{1}{2} Tr (a(x) D^2u (x)) + b(x)\cdot D u(x),
\;\;\; x \in \mathbb{R}^d,
$$
with $a(x) = \sigma(x) \sigma^*(x)$,   assuming
Hypotheses \ref{hy1}, \ref{hy2} and \ref{hy3}.

\vskip 1mm The next  result provides new   estimates
for $L$ in H\"older spaces.  These estimates are
not covered
by recent papers dealing with elliptic and parabolic equations with
unbounded coefficients, due to the fact that in our case also
   {\it $f$ can be  unbounded} (compare with \cite{Ce},
\cite{BL}, \cite{KP} and the references therein).

\begin{theorem} \label{bbo}
 Let $\theta \in (0,1)$.   For any $\theta' \in
 (0,\theta)$, there exists  $\lambda_0>0$ (depending on
 $\theta, \theta ' , d,$ $[b]_{\theta}$, $\| \sigma \|_0$,
  $\| a^{-1}\|_0$,  $\| D^k \sigma\|_0,$
  $k=1, 2, 3$) such that, for $\lambda \ge \lambda_0$,
    for any
 $f \in C^{\theta} (\mathbb{R}^d)$, the
equation
\begin{equation} \label{u}
\lambda u - L u = f
\end{equation}
admits a unique classical solution
$u = u_{\lambda}\in C^{2+ \theta'} (\mathbb{R}^d)$ for which
\begin{equation}
\label{sh} \| u\|_{2+ \theta \, '} = \| u(\cdot) \, (1+
|\cdot|)^{-1} \|_0 +
  \|Du\|_0  + \|D^2u\|_0  + [D^2 u ]_{\theta\, '} \le
C(\lambda) \| f \|_{\theta }
\end{equation}
with $C(\lambda)$ (independent on $u$ and $f$)  such that
$C(\lambda)$ $\to 0$
as $\lambda \to + \infty$.
\end{theorem}
\begin{proof}
Uniqueness can be proved  by the following argument (cf.
\cite[page 606]{Kry2}). Consider
  $\eta (x) = \sqrt{1+ |x|^2}$, $x \in \mathbb{R}^d$.

Defining $u =v \eta $, we  obtain an elliptic equation for the
bounded function $v$, i.e.,
$$
\lambda v(x)  - \frac{1}{2} Tr (a(x) D^2v (x)) -
(b(x) + \frac{a(x) D \eta(x)}{\eta(x)}) \cdot D v(x)
$$
\begin{equation} \label{v}
- \Big(\frac{1}{2} \frac{ Tr(a(x) D^2 \eta(x))}{\eta(x)}  +
b(x) \cdot \frac{ D \eta(x)}{\eta(x)} \Big) v(x) = \frac{f(x)}
{\eta(x)},
\end{equation}
$x \in \mathbb{R}^d$. Note that $v$ has first and second  bounded
derivatives. For $\lambda$ large enough (depending on
 $\| \sigma\|_0$ and  $\| \frac{b}{\eta}\|_0$),
  uniqueness of $v$
follows by the classical maximum principle.

\medskip Now
we divide the rest of the  proof in some   steps.

\medskip \noindent {\bf Step I.}  \ We assume {\it
in addition that $b
\in C^3 (\mathbb{R}^d, \mathbb{R}^d)$ and has all bounded
derivatives
up to the third order}
(but it is not necessarily bounded). We prove that, for sufficiently
large $\lambda>0$, there exists a unique solution $u = u_{\lambda}
\in C^{2+ \theta } (\mathbb{R}^d)$ to the equation
$$
\lambda u - L u =   f \in C^{\theta} (\mathbb{R}^d).
$$
Moreover there exists $C$ (independent on $u$ and $f$) such that
\begin{equation}  \label{gh}
\| u \|_{2+  \theta} \le C  \| f \|_{\theta}.
\end{equation}
Estimates \eqref{gh} are new Schauder estimates since $f$
is not assumed to be bounded (compare with \cite{Ce}
 and \cite{BL})

We consider  the
function
\begin{equation} \label{fr}
u (x) = \int_0^{\infty} e^{- \lambda t } \mathbb{E} [f(X_t^x)] dt =
\int_0^{\infty} e^{- \lambda t } P_t f(x)dt, \;\; x \in \mathbb{R}^d,
\end{equation} where $(X_t^x)$ is the  solution of \eqref{SDE}
 and show that,
for $\lambda$ large enough, $u$ is
a $C^{2+ \theta}(\mathbb{R}^d)$-solution  to our PDE.

\smallskip
Using that
$\mathbb{E} |X_t^x - X_t^y| \le C e^{Ct} |x-y| $,
 $t \ge 0$, $x , y \in \mathbb{R}^d$, we find
$$
|u(x) - u(y)| \le c [f]_{\theta,1} \, (|x-y|^{\theta} \vee
|x-y|),\;\; x , y \in \mathbb{R}^d,
$$
and also $\|u(\cdot) \, (1+ |\cdot |)^{-1} \|_0 \le C
  \|f(\cdot) \, (1+ |\cdot |)^{-1} \|_0  $, for $\lambda$
   large enough.

By Lemma \ref{semi} we get, for $\lambda$ large enough,
$$
\|Du  \|_0 + \|D^2 u  \|_0  \le C [f]_{\theta}.
$$
To estimate the second derivatives of $u$, we proceed
as in \cite[Theorem 4.2]{Pr2}. We have, for any $x, y \in \mathbb{R}^d$
 with $|x-y| \le 1$,
$$
|D^2u(x) - D^2u(y)| = \int_0^{|x-y|^2} e^{- \lambda t } |D^2P_t f(x)
- D^2P_t f(y)|  dt
$$
$$
+ \int_{|x-y|^2}^{\infty} e^{- \lambda t } |D^2P_t f(x) - D^2P_t
f(y)| dt
$$
$$
\le c'' |x-y|^{\theta} [f]_{\theta} +\,
C|x - y|  \, [f]_{\theta} \,  \int_{|x-y|^2}^{\infty} e^{-
\lambda t }  \, \frac{e^{c  t}}{t^{(3 -\theta)/2}} \; dt \le c'
[f]_{\theta} |x-y|^{\theta}.
$$
It remains to check that $u$ is a solution. This is not difficult
thanks to Lemma \ref{semi}
(see, for instance,  \cite[Chapter 1]{Ce} or argue as in
\cite[Theorem 4.1]{Pr2}).

\medskip \noindent {\bf Step II.} Under the assumptions of
Step I,   for any $\alpha \in (0, \theta)$,  we have
\begin{equation}  \label{gh1}
\| u \|_{2+  \alpha } \le C(\lambda)  \| f \|_{\theta},
\end{equation}
with $C(\lambda) \to 0 $, as $\lambda \to + \infty$.
This  is clear if we replace $\| u\|_{2+ \alpha}$
 with  $\displaystyle{  \| u(\cdot) \, (1+
|\cdot|)^{-1} \|_0 }$ $+
  \|Du\|_0  + \|D^2 u \|_{0}$. Therefore, we  only consider
   $[D^2 u]_{\alpha}.$

Combining  the interpolatory estimate: $[v]_{\alpha} \le
C \| v\|_0^{1 - {\alpha}}$
$\|Dv\|^{\alpha}_0$, $v \in C^{1}_b (\mathbb{R}^d)$
 (where $C =C(d)$, see \cite[Section 3.2]{Kr})
with estimates of Lemma \ref{semi} corresponding to  $j=2,3$,
  we find, for any $t>0$,
$$
[D^2 P_t f]_{\alpha} \le  C \|D^2 P_t f \|_0^{1-\alpha} \,
\|D^3 P_t f \|_0^{\alpha} \le
C_4 [f]_{\theta}  \, \frac{e^{c_4 t}}{t^{\gamma}},
$$
with $\gamma = \frac{ 2 - \theta  + \alpha}{2 } <1$
(since $\alpha < \theta$).
It follows
$$
[D^2 u ]_\alpha \le C_4 [f]_{\theta}
\int_0^{+\infty} \frac{e^{(c_4 -\lambda)  t}}  { t^{\gamma}
}  dt \le C_5 [f]_{\theta} (\lambda- c_4)^{\gamma-1}.
$$
The assertion is proved.

\vskip 2mm  \noindent {\bf Step III.} We require that
$b \in C^{\theta} (\mathbb{R}^d, \mathbb{R}^d)$ as in Hypothesis
 \ref{hy1}
 and prove the following
a-priori estimates: if $\lambda $ is large enough and
  $u \in C^{2+ \theta' }(\mathbb{R}^d)$,
  $0 < \theta' < \theta $, is a
solution to  $
\lambda u - L u =   f \in C^{\theta}(\mathbb{R}^d)$, then
\begin{equation} \label{sh1}
\| u(\cdot) \, (1+
|\cdot|)^{-1} \|_0 +  \| D  u\|_{0} + \| D^2
u \|_{\theta'}
 \le K(\lambda) \| f \|_{\theta},
\end{equation}
with  $K(\lambda ) \to 0$, as $\lambda \to
+\infty.$

To prove the  estimate
 we introduce $\rho \in C_0^{\infty}(\mathbb{R}^d)$, $0
\le \rho \le 1 $, $\rho(x) = \rho(-x)$, for any $x \in
\mathbb{R}^d,$
 $\int \rho(x)\, dx =1$. Moreover, $b * \rho$ indicates $b$ convoluted with
$\rho$.

\medskip Write $ \lambda u(x) -
\frac{1}{2} Tr (a(x) D^2u (x)) - (b * \rho) (x)\cdot D u(x) =
f(x) +
 \langle \big( b - (b* \rho) \big)(x), Du(x)\rangle
$. It is easy to see that $b * \rho$  (even if it can be unbounded)
is a $C^{\infty}-$function with all bounded derivatives.
Moreover,  there exists $C = C (\theta, D \rho, D^2 \rho,
 D^3 \rho)>0$ such that
\begin{align} \label{roo}
\| D^k (b* \rho) \|_0 \le C [b]_{\theta},\;\;\; k=1,2,3.
\end{align}
The function $b - (b* \rho)$ is bounded  and we have
$$
\| b - (b* \rho)\|_0 \le C [b]_{\theta}.
$$
It follows that  $ b - (b* \rho) \in C^{\theta}_b (\mathbb{R}^d,
\mathbb{R}^d)$.
Applying Step II,  we find  that
\begin{equation} \label{df}
 \| u\|_{2+ \theta '}
\le C(\lambda)  \| f\|_{\theta} +
C (\lambda)\big \| \langle b - (b* \rho), Du\rangle
 \big \|_{\theta}
\end{equation}
with $C(\lambda) \to 0$. Using that
$$
\| \langle b - (b* \rho), Du\rangle
 \|_{\theta} \le c [ b ]_{\theta} \|Du
\|_0
+ c  [ b ]_{\theta} \| Du\|_{\theta}
\le   c [b]_{\theta} \| u\|_{2 + \theta \, '},
$$
for some constant $c$ depending on $\theta$, we rewrite \eqref{df}:
$$ \| u\|_{2+ \theta'}
\le C(\lambda)  \| f\|_{\theta} +
C (\lambda) c [b]_{\theta } \| u\|_{2 + \theta \, '}.
$$
Choosing $\lambda_0 >0$  such that  $C(\lambda) < \frac{1}{c \,
[b]_{\theta}}$, for $\lambda \ge \lambda_0$, we find, with $u =
u_{\lambda}$
 \begin{equation} \label{apri}
(1 - C (\lambda)  c_{} [b]_{\theta} )\,
\| u\|_{2+ \theta \, '}
\le C(\lambda)  \| f\|_{\theta}.
\end{equation}
Defining $K(\lambda ) = \frac{C(\lambda ) }{1 - C (\lambda)  c_{}
[b]_{\theta}
  }  $, we get the assertion.

\vskip 2mm  \noindent {\bf Step IV.} We show that for $\lambda \ge
\lambda_0 $ (see Step III) there exists
a classical solution $u = u_{\lambda }
\in C^{2 + \theta '} (\mathbb{R}^d)$ to \eqref{u}.
This assertion will conclude the proof.

\smallskip  We fix  $\lambda \ge \lambda_0$.
To prove the result, we will use the
continuity method. To this purpose, using the test function $\rho$
of Step III, we consider:
\begin{equation} \label{drr}
 \lambda u(x) -   \frac{1}{2} Tr (a(x) D^2u (x)) -  (1 - \delta)
(b
* \rho) (x) \cdot D u(x) -  \delta b(x)\cdot D u(x) = f(x),
\end{equation}
$x \in \mathbb{R}^d,$ where $\delta \in [0,1]$ is a parameter. Let
us define
$$
\Gamma = \{  \delta  \in [0,1] \, :\, \mbox{eq. \eqref{drr}
has a unique solution $u =u_{\delta}
 \in C^{2+ \theta '}(\mathbb{R}^d)$},
  \mbox{ for any} \;
  f \in C^{\theta}(\mathbb{R}^d)\}.
$$
$\Gamma $ is not empty since $0 \in \Gamma $ by Step I.
Let us fix $\delta_0 \in \Gamma$ and
rewrite equation \eqref{drr} corresponding to an arbitrary
$\delta \in [0,1]$ as
$$
 \lambda u(x) -   \frac{1}{2} Tr (a(x) D^2u (x)) -  (1 - \delta_0)
(b* \rho) (x) \cdot D u(x) -  \delta_0 b(x)\cdot D u(x)
$$
$$
= f(x) + [\delta - \delta_0] \, ( b - \, b* \rho )(x) \cdot D
u(x).
$$
Introduce the operator ${\cal T} : C^{2+ \theta ' }(\mathbb{R}^d)
\to
C^{2+ \theta ' }(\mathbb{R}^d)$. For any $v \in C^{2+ \theta ' }(\mathbb{R}^d)
$, ${\cal T} v =u $ is the (unique)
$C^{2+ \theta ' }(\mathbb{R}^d)$-function which solves
$$
 \lambda u(x) -   \frac{1}{2} Tr (a(x) D^2u (x)) -  (1 - \delta_0)
(b* \rho) (x) \cdot D u(x) -  \delta_0 b(x)\cdot D u(x)
$$
$$
= f(x) + [\delta - \delta_0] \, ( b - \, b* \rho )(x) \cdot  D
v(x).
$$
Using the a-priori estimates \eqref{apri}, we get that
$$
\| {\cal T} v -  {\cal T} w  \|_{2+ \theta'} \le 2 K(\lambda)
|\delta - \delta_0| \, [b]_{\theta} \,\| v - w \|_{2 + \theta '},
\;\; \; v , w \in C^{2+ \theta ' }(\mathbb{R}^d).
$$
Choosing $|\delta - \delta_0|$ small enough, the operator $\cal T$
becomes a contraction on $C^{2+ \theta ' }(\mathbb{R}^d)$ and
it has a unique fixed point which is the solution to
 \eqref{drr}. Therefore for $|\delta - \delta_0|$ small
  enough, we have that $\delta \in \Gamma$.
     A compacteness argument
  shows that $\Gamma = [0,1]$. The assertion is proved.
\end{proof}

\section{Differentiable stochastic flow  }

{Given }  $x\in{\mathbb{R}}^{d}$, consider the stochastic
differential equation in ${\mathbb{R}}^{d}$ :
\begin{equation}
dX_{t}=b\left( X_{t}\right)  dt + \sigma (X_t) dW_{t},\quad \quad
X_{s}=x,\;\;\; t \ge s \ge 0. \label{SDE1}
\end{equation}
As already mentioned  our key result is the existence of a
\emph{differentiable} stochastic flow
$(x,s,t)\mapsto\varphi_{s,t}(x)$ for equation~(\ref{SDE1}). Recall
the relevant definition from H. Kunita~\cite{K}:

\begin{definition}
A \emph{stochastic flow of diffeomorphisms} (resp. \emph{of class}
$C^{1,\alpha}$) on the stochastic basis $\left(  \Omega,\left( \mathcal{F}{}_{t}\right)
,{}\mathcal{F},P,\left(  W_{t}\right) \right)  $ associated to
equation (\ref{SDE1}) is a map
$(s,t,x,\omega)\mapsto\phi_{s,t}(x)\left(  \omega\right) $, defined
for $0\leq s\leq t $, $x\in{\mathbb{R}}^{d}$, $\omega\in \Omega$
with values in ${\mathbb{R}}^{d}$, such that

\begin{itemize}
\item [(a)]given any $s \ge 0  $, $x\in{\mathbb{R}}^{d}$, the
process $X^{s,x}=\left( X_{t}^{s,x}\left(  \omega\right) ,t\ge s
,\omega\in\Omega\right)  $ defined as $X_{t}^{s,x}=\phi_{s,t}(x)$ is
a continuous $\mathcal{F}_{s,t}$-measurable solution of equation
(\ref{SDE1});

\item[(b)] $P$-a.s., for all
$0\leq s\leq t $, $\phi_{s,t}$ is a diffeomorphism  and the
functions $\phi_{s,t}(x)$, $\phi_{s,t}^{-1}(x)$, $D\phi_{s,t}(x)$,
$D\phi_{s,t}^{-1}(x)$ are continuous in $(s,t,x)$ (resp. of class
$C^{1,\alpha}$ in $x$ uniformly in $(s,t)$,  for $0\leq s\leq t \le
T$, with $T>0$);

\item[(c)] $P$-a.s., $\phi_{s,t}(x)
=\phi_{u,t}(\phi_{s,u}(x))$, for all $0\leq s\leq u\leq t $,
$x\in{\mathbb{R}}^{d}$, and $\phi_{s,s}(x)=x$.
\end{itemize}
\end{definition}

\vskip 2mm
Starting from the work of Zvonkin, an important approach to the
analysis of SDEs with non-regular drift is based on the
transformation $\Psi _{t}:\mathbb{R}^{d}\rightarrow\mathbb{R}^{d}$,
solution of the vector valued equation
\[
\frac{\partial\Psi_{t}}{\partial t}+L\Psi_{t}=0\text{ on }\left[
0,T\right]  ,\quad\Psi_{T}\left(  x\right)  =x
\]
where $\Psi_{t}(x)= \Psi(t,x)$ and
  $\left[  0,T\right]  $ is a time
interval where the SDE is considered. At time $T$, the solution is
an isomorphism by definition;\ one has to prove suitable regularity
and invertibility of $\Psi_{t}$ for $t\in\left[
0,T\right]  $. Then $Y_{t}:=\Psi_{t}\left(  X_{t}\right)  $ satisfies%
\[
dY_{t}=D\Psi_{t}\left(  \Psi_{t}^{-1}\left(  Y_{t}\right)  \right)
\sigma\left(  \Psi_{t}^{-1}\left(  Y_{t}\right)  \right)  dW_{t}.
\]
The irregular drift has been removed. This approach, although
successful (see \cite{Ba},
\cite{GM},  \cite{Kry-Ro},  \cite{Za},
\cite{Za2}), raises two delicate questions: i) one has to deal with
unbounded initial conditions;\ ii) one has to prove some form of
invertibility.

We propose a variant, based on the same operator $L$ but on the
vector valued equation
\[
\lambda\psi-L\psi=b
\]
(under other assumptions one can treat also the time-dependent case
through the parabolic equation
$\lambda\psi_{t}-\frac{\partial\psi_{t}}{\partial t}-L\psi_{t}=b$,
see \cite{FGP}). We find it more tractable than the case of
unbounded initial condition; and we translate the difficult
invertibility issue in the smallness of the gradient of the
solution, obtained by means of a large $\lambda$. When the gradient
of $\psi$ is less than one, the function $\Psi\left( x\right)
=x+\psi\left(  x\right)  $ is invertible and the
process $Y_{t}:=\Psi\left(  X_{t}\right)  $ satisfies%
\[
dY_{t}= D\Psi\left(  \Psi^{-1}\left(  Y_{t}\right) \right)
\sigma\left(  \Psi^{-1}\left(  Y_{t}\right) \right)
dW_{t}+\lambda \psi\left(  \Psi^{-1}\left(  Y_{t}\right) \right)  dt .
\]
So, at the end, the
transformed equation has the same degree of difficulty as in the
case of the Zvonkin's transformation.

\begin{theorem}
\label{th:flow1} Assume  Hypotheses \ref{hy1}, \ref{hy2},
\ref{hy3} and fix  any $\theta '' \in (0, \theta)$.
Then we have the following facts:

\begin{itemize}
\item [(i)](pathwise uniqueness) For every $s\ge 0 $,
$x\in{\mathbb{R}}^{d}$, the stochastic equation (\ref{SDE1}) has a
unique continuous adapted solution $X^{s,x}=\left( X_{t}^{s,x}\big(
\omega\right) ,t\ge s,$ $\omega\in\Omega\big) $.

\item[(ii)] (differentiable flow) There exists a stochastic flow
$\phi= (\phi_{s,t})$ of diffeomorphisms for equation (\ref{SDE1}).
The flow is also of class $C^{1,{\theta}^{''}}$.

\item[(iii)] (stability) Let $(b^{n})\subset
C^{{\theta}}(\mathbb{R}^d, \mathbb{R}^d)$  and let $(\phi^{n})$ be
the corresponding stochastic flows.
Assume that there exists $b \in C^{{\theta}}(\mathbb{R}^d,
 \mathbb{R}^d)$ such that $b_n - b \in C^{{\theta}}_b(\mathbb{R}^d,
 \mathbb{R}^d)$, $n \ge 1$, and $\|b - b_n \|_{C^{{\theta}}_b}
  \to 0$ as $n \to \infty$.
 If  $\phi$ is the  flow associated to
$b$, then, for any
$p\geq 1$, $T> 0$,
\begin{equation}
\label{stability1}
\lim_{n\rightarrow\infty}\sup_{x\in{\mathbb{R}}^{d}} \sup_{0 \le
s\le T} E[ \sup_{u \in [s,T]}
\frac{|\phi_{s,u}^{n}(x)-\phi_{s,u}(x)|^{p}}{(1+|x|)^p}]=0.
\end{equation}%
\begin{equation}
\sup_{n\in \mathbb{N}}\sup_{x\in {\mathbb{R}}^{d}}\sup_{0\leq s\leq
T}E[\sup_{u\in \lbrack s,T]}\Vert D\phi _{s,u}^{n}(x)\Vert
^{p}]<\infty , \label{bound}
\end{equation}
\begin{equation}
\label{stability2}
\lim_{n\rightarrow\infty}\sup_{x\in{\mathbb{R}}^{d}}\sup_{0\leq
s\leq T} E[  \sup_{u \in [s,T]} \Vert
D\phi_{s,u}^{n}(x)-D\phi_{s,u}(x)\Vert^{p}]=0.
\end{equation}
($\Vert\cdot\Vert$ denotes the Hilbert-Schmidt norm).
\end{itemize}
\end{theorem}

\begin{proof}
\textbf{Step 1} (auxiliary elliptic  systems).
Let us choose $\theta '$ such that  $0< \theta ''
 <  \theta' < \theta $.

For a fixed
$\lambda \ge \lambda_0 >0$
 (see Theorem \ref{bbo})
 we consider the unique classical
 solution $\psi = \psi_{\lambda} \in C^{2+ \theta ' }
 (\mathbb{R}^d, \mathbb{R}^d)$
  to the elliptic
system
\begin{equation} \label{pde}
\lambda\psi_{\lambda } -  L\psi_{\lambda} =b,\;\;\;
\end{equation}
where
$$
Lu(x)=\frac{1}{2} Tr ( \sigma(x) \sigma^*(x) D^2 u(x) ) +
b(x)\cdot D u(x),
$$
for any smooth function
$u: {\mathbb{R}}
^{d}\rightarrow{\mathbb{R}}^{d}$ (clearly
\eqref{pde} has to be interpreted
componentwise).

Define
\[
\Psi_{\lambda}(x)=x+\psi_{\lambda}(x).
\]

Similarly to  \cite[Lemma 8]{FGP} we have

\begin{lemma} \label{diff} For  $\lambda $ large enough, such that
$\Vert D \psi_{\lambda}\Vert_{0} <1$ (see
Theorem  \ref{bbo}),
the following statements hold:

{{\vskip 2mm \noindent }} (i)  $\Psi_{\lambda}$ has bounded first
and  spatial
derivatives and moreover  the second (Fr\'echet) derivative $D^2_x
\Psi_{\lambda}$ is globally $\theta '$-H\"older continuous.

{{\vskip 2mm \noindent }} (ii) \
$\Psi_{\lambda}$ is
a $C^2$-diffeomorphism of $\mathbb{R}^d$.

{{\vskip 2mm \noindent }} (iii)  $\Psi_{\lambda}^{-1}$ has bounded
first and second
derivatives and moreover
\begin{align}\label{gra1}
D\Psi _{\lambda }^{-1}(y)= \sum_{k\geq 0}\big(-D\psi _{\lambda }
(\Psi _{\lambda }^{-1}(y))\big)^{k},\;\;\;y\in \mathbb{R}^{d}.
\end{align}
\end{lemma}

{\sl In the sequel we will use a value of $\lambda$ for which Lemma \ref{diff} holds
and simply write $\psi$ and $\Psi$ for $\psi_{\lambda}$ and
$\Psi_{\lambda}$.}

\vskip 2mm \textbf{Step 2} (conjugated SDE). Define
\[
\widetilde{b}(y)= \lambda\psi(\Psi^{-1}(y)),\quad\widetilde{\sigma
}(y)=D\Psi(\Psi^{-1}(y)) \, \sigma ( \Psi^{-1}(y) )
\]
and consider, for every $s\ge 0  $ and $y\in{\mathbb{R}}^{d}$, the
SDE
\begin{equation}
Y_{t}=y+\int_{s}^{t}\tilde{\sigma}(Y_{u})dW_{u}+
\int_{s}^{t}\widetilde {b}(Y_{u})du,\qquad t\ge s. \label{conjugated
SDE}
\end{equation}
This equation is equivalent to equation (\ref{SDE1}), in the
following sense.\ If $X_{t}$ is a solution to \eqref{SDE1}, then
$Y_{t}=\Psi(X_{t})$ verifies equation (\ref{conjugated SDE}) with
$y=\Psi(x)$: it is sufficient to apply It\^{o} formula to
$\Psi(X_{t})$ and use equation (\ref{pde}).

Viceversa, given a solution $Y_{t}$ of equation (\ref{conjugated
SDE}), let $X_{t}=\Psi^{-1}(Y_{t})$, then it is possible to prove by
direct application of It\^{o} formula that $X_{t}$ is a solution of
\eqref{SDE1} with $x=\Psi^{-1}(y)$. This is not very important since below
we will obtain this fact indirectly.

\smallskip
\textbf{Step 3} (proof of (i) and (ii)). We have  clearly
$\widetilde{b}$ and $\widetilde{\sigma}\in C^{1+{\theta'}}$
 (with first order derivatives bounded and in $C^{\theta'}_b)$
 so that, in particular, they are Lipschitz continuous.

By classical results (see \cite[Chapter 2]{K}) this implies
existence and uniqueness of a strong solution $Y$ of equation
(\ref{conjugated SDE}) and even the existence of a
$C^{1,\theta^{''}}$  stochastic flow of diffeomorphisms
$\varphi_{s,t}$ associated to equation (\ref{conjugated SDE}).

The uniqueness of $Y$ implies the pathwise
uniqueness of solutions of the original SDE \eqref{SDE} since two solutions $%
X,\tilde{X}$ give rise to two processes $Y_{t}=\Psi (X_{t})$ and $\tilde{Y}%
_{t}=\Psi (\tilde{X}_{t})$ solving~\eqref{conjugated SDE}, then $Y=\tilde{Y%
}$ and then necessarily $X=\tilde{X}$. By the Yamada-Watanabe
theorem pathwise uniqueness together with weak existence (which is a
direct consequence of the Girsanov formula) gives the existence of
the (unique) solution $(X_{t}^{x})_{t\geq s}$ of eq. \eqref{SDE}
starting from $x$ at time $s$. Moreover setting
$$\phi
_{s,t}=\Psi^{-1}\circ \varphi
_{s,t}\circ \Psi
$$ we realize that $\phi _{s,t}$ is the flow of~%
\eqref{SDE} (in the sense that $X_{t}^{x}=\phi _{s,t}(x)$,
$P$-a.s.).

\smallskip \textbf{Step 4.} (proof of (iii)). Let $\psi ^{n}$ and
$\psi$ be the
solutions in $C^{2+\theta' }({\mathbb{R}}^{d};{%
\mathbb{R}}^{d})$ respectively
of the elliptic  problem associated
to
$b_{n}$ and
to $b \in C^{\theta }({\mathbb{R}}^{d};{ \mathbb{R}}^{d})$.
 Notice that we
can make a choice of $\lambda $ independent of $n$.  We write
\begin{equation*}
\lambda \left( \psi ^{n}-\psi \right)-L\left( \psi ^{n}-\psi \right)
=\left( b^{n}-b\right) +\left( b^{n}-b\right) \cdot D\psi ^{n},
\;\;\; n \ge 1.
\end{equation*}
By Theorem~\ref{bbo} we have
 $\sup_{n \ge 1}\| \psi_n\|_{C^{2+\theta' }} \le C < \infty$.
Since $b - b_n$ is a bounded function, by the classical maximum
principle (see \cite{Kr})
 we infer also that $\psi - \psi_n$ is a bounded function on $\RR^d$
 and
  \begin{align} \label{max}
 \| \psi - \psi_n \|_0 \le \frac{C +1}{\lambda}
 \| b - b_n \|_0,\;\;\; n \ge 1.
 \end{align}
It follows that $\psi - \psi_n \in
C^{2 + \theta' }_b({\mathbb{R}}^{d};{ \mathbb{R}}^{d})$
and $\| \psi - \psi_n  \|_{C^{2 + \theta' }_b}$ $\to 0$ as $n \to
 \infty$.

 Fix $p \ge 1$ and consider the flows $\varphi
_{s,t}^{n}=\Psi^{n}\circ \phi _{s,t}^{n}\circ (\Psi^{n})^{-1}$ which
satisfy
\begin{equation}
\varphi_{s,t}^{n}(y)=y+\int_{s}^{t}\widetilde{b}_{}^{n}\circ\varphi
_{s,u}^{n}(y)du+\int_{s}^{t}\widetilde{\sigma}_{}^{n}\circ\varphi
_{s,u}^{n}(y)\cdot dW_{u},
\end{equation}
We have $\widetilde{\sigma }^{n}\rightarrow \widetilde{\sigma }$ and $%
\widetilde{b}^{n}\rightarrow \widetilde{b}$, as $n \to \infty$,
 in
$C^{1+\theta' }({\mathbb{R}}^{d};{\mathbb{R}}^{d\times k}) $
and $C^{1+\theta' }({\mathbb{R}}^{d};{\mathbb{R}}%
^{d})$, respectively. By standard arguments, using the Gronwall
lemma, the Doob inequality and the Burkholder inequality (compare,
for instance, with the proof of \cite[Theorem II.2.1]{K}) we obtain
the analog of \eqref{stability1} for the auxiliary flows $\varphi
_{s,t}^{n}$ and $\varphi _{s,t}$:
\begin{equation}
\label{stability1-bis}
\lim_{n\rightarrow\infty}\sup_{x\in{\mathbb{R}}^{d}} \sup_{0 \le
s\le T} E[ \sup_{u \in [s,T]}
\frac{|\varphi_{s,u}^{n}(x)-\varphi_{s,u}(x)|^{p}}{(1+|x|)^p}]=0.
\end{equation}%
 We  can also
prove the inequality
\begin{equation}
\label{stability20-bis}
\sup_{n\in \mathbb{N}}\sup_{x\in {\mathbb{R}}^{d}}\sup_{0\leq s\leq
T}E[\sup_{u\in \lbrack s,T]}\Vert D\varphi _{s,u}^{n}(x)\Vert
^{p}]<\infty ,
\end{equation}
for $D\varphi _{s,t}^{n}(y)$,   using the fact that the stochastic
equation for $D\varphi _{s,t}^{n}(y)$ has the identity as initial
condition and  random coefficients $D\widetilde{b}^{n}\left( \phi
_{s,u}^{n}\right) $ and $D\widetilde{\sigma }^{n}\left( \phi
_{s,u}^{n}\right) $ which are uniformly bounded functions (since $\|
D \widetilde{b}^{n} \|_0$ $ +\| D\widetilde{\sigma }^{n} \|_0 \le
C$, uniformly in $n$).

To prove~\eqref{bound} is then enough to estimate $D\phi _{s,u}^{n}$
using~\eqref{stability20-bis}, the uniform boundedness of the
derivatives of $\Psi^{n}$ and its inverse (note that the
uniform boundedness of the $D(\Psi^{n})^{-1}$ can be proved
by   \eqref{gra1}).

To prove~\eqref{stability1} we remark that to estimate the
difference $\varphi _{s,t}^{n}(\Psi^n(x))-\varphi _{s,t}(\Psi(x))$
we can split it as $ \varphi _{s,t}^{n}(\Psi^{n}(
x))-\varphi_{s,t}(\Psi^{n}( x)) +\varphi_{s,t}(\Psi^{n}( x))-\varphi
_{s,t}(\Psi( x)) . $ The two differences can then be controlled by
$$
\mathbb{E}[\sup_{s\le u \le T}|\varphi _{s,u}^{n} (\Psi^{n}(
x))-\varphi_{s,u} (\Psi^{n}( x))|^p]\le a_n \,
 (1+|\Psi^{n}(
x)|)^p\le
a_n \, (1+| x|)^p,
$$
(where $a_n = \sup_{x\in{\mathbb{R}}^{d}} \sup_{0 \le s\le T} E[
\sup_{u \in [s,T]}
\frac{|\varphi_{s,u}^{n}(x)-\varphi_{s,u}(x)|^{p}}{(1+|x|)^p}]$ and
$\lim_{n \to \infty} a_n =0$) and by
$$
\mathbb{E}[\sup_{s\le u \le T}|\varphi_{s,u}(\Psi^{n}( x))-\varphi
_{s,u}(\Psi( x))|^p]\le \sup_{z\in\mathbb{R}^d}\mathbb{E}[\sup_{s\le u \le T}\|D\varphi_{s,u}(z)\|^p] |\Psi^n(x)-\Psi(x)|^p
$$
$$
\le C \|\Psi^n-\Psi\|_{0}^p,
$$
with $\lim_{n \to \infty}  \|\Psi^n-\Psi\|_{0}
= \lim_{n \to \infty}  \|\psi^n-\psi\|_{0}
 =0$ (see \eqref{max}).

Finally, one has to check that $ (\Psi^{n})^{-1}$ converges to
$\Psi^{-1}$ in the supremum norm.
This follows from the inequality
$$
\sup_{y \in \RR^d}|(\Psi^{n})^{-1} (y) - \Psi^{-1}(y)|
\le \sup_{x \in
\RR^d}
 | (\Psi^{n})^{-1}(\Psi^{n} (x) ) - \Psi^{-1}(\Psi^n (x) )|
$$$$ \le \sup_{x \in \RR^d}
 | \Psi^{-1}(\Psi^{n} (x) ) - \Psi^{-1}(\Psi (x) )|
  \le  \| D \Psi^{-1} \|_0
   \,  \| \Psi -  \Psi^n  \|_0,
$$
which tends to 0, as $n \to \infty$.

Arguing as in the proof of \cite[Theorem II.3.1]{K}, we get the
following linear equation for the derivative $D\phi _{s,t}(x)$
\begin{equation} \label{dove}
\begin{split}
\lbrack D\Psi_{}(\phi_{s,t}(x))] D\phi_{s,t}(x)    =D\Psi_{}(x)
 +\int_{s}
^{t}[D^{2}\Psi_{} (\phi_{s,u}(x))]D\phi_{s,u}(x)
\, \sigma (\phi_{s,u}(x))  dW_{u}\\
+ \int_{s} ^{t} D\Psi_{} (\phi_{s,u}(x))
\, [D \sigma (\phi_{s,u}(x))] D\phi_{s,u}(x)   dW_{u}
-\lambda\int_{s}^{t}[D\psi_{}(\phi_{s,u}(x))]D\phi_{s,u}(x)du,
\end{split}
\end{equation}
$0\leq s\leq t\leq T$, $x\in \mathbb{R}^{d}$. From the
fact that
$ \lim_{n \to \infty} \| \psi ^{n} - \psi \|_{C^{2+\theta' }_b} =0 $
together with \eqref{bound} and \eqref{dove}, we finally obtain
\begin{equation}
\lim_{n\rightarrow \infty }\sup_{x\in {\mathbb{R}}^{d}}\sup_{0\leq
s\leq T}E[\sup_{u\in \lbrack s,T]}\Vert D\phi _{s,u}^{n}(x)-D\phi
_{s,u}(x)\Vert ^{p}]=0,  \label{stima1}
\end{equation}
which concludes the proof.
\end{proof}

\bigskip

We consider now two possible extensions
of Theorem \ref{th:flow1} to
the case when  coefficients
$b$ and $\sigma_i$ are  time-dependent continuous  functions
defined on
 $[0,T] \times \RR^d$, i.e., we are dealing with
\begin{equation}\label{fy}
dX_{t}^{x}=b\left( t, X_{t}^{x}\right)
dt+\sum_{i=1}^{k}\sigma_{i}\left(t, X_{t}^{x}\right) dW_{t}^{i},
\;\;\; t\in\left[  0,T\right]  ,\quad
X_{0}=x.
\end{equation}

\begin{remark} \label {vi}   {\em
 Let us
  treat the case in which also
 $b$ is  {\it bounded.}
 Following \cite{FGP}, an analogous of our  Theorem
\ref{th:flow1} holds for \eqref{fy} if  we  require that
$b$ and $\sigma_i$
are continuous and bounded functions such that
 $$
 \sup_{t \in [0,T]} (\| b(t , \cdot) \|_{C^{\theta}_b}
 + \, \| \sigma_i (t ,
  \cdot) \|_{C^{1+ \theta}_b} ) < \infty,\;\;\; i=1, \ldots, k,
  $$
 and, moreover (as in Hypothesis \ref{hy3}) we  assume
  that $ \sigma (t,x)$ is uniformly non-degenerate, i.e.,
there exists the inverse
 of $a(t,x)= \sigma (t,x)\sigma^* (t,x)$, for any
  $t \in [0,T]$, $x \in \RR^d$, and
\begin{align}\label{f7}
\| a^{-1}\|_{0}  =  \sup_{x \in \mathbb{R}^d, \, t \in [0,T]}
\| a^{-1} (t,x) \| < \infty.
\end{align}
To prove Theorem \ref{th:flow1} under these hypotheses,
 one can   follow the proof
of the analogous result proved in \cite{FGP}.
We only give a sketch of the argument.

First note that  \cite[Theorem 2]{FGP} remains the same even with
the
previous non-constant
 $\sigma = (\sigma_i)$
  (indeed it is a special case of a result in \cite{KP}).
 Then  \cite[Lemma 4]{FGP} is true  with  $\sigma$ in \eqref{fy}
by the following
 rescaling
argument.
Consider
 $\lambda \ge 1$ and
$$
\partial_{t}u_{\lambda}+ L u_{\lambda} -
\lambda u_{\lambda}=f\;\;\; \text{in} \;\; [0,\infty)\times
{\mathbb{R}}^{d},
$$
where $L $ is the Kolmogorov operator associated to the SDE, i.e.,
$$L =
\frac{1}{2} \mathrm{Tr} [a(t,x)  D^2 u(t,x)]
 + b (t,x) \cdot D
u(t,x)$$ (here $( \sigma (t,x) \sigma^*(t,x))  = a(t,x)$ and $D$
and $D^2$
denote spatial derivatives).  Define a function $v$
on $[0,\infty)\times {\mathbb{R}}^{d}$ such
that $
v(\lambda t, \sqrt{\lambda}\,  x) = u_{\lambda} (t,x)$, $t \ge
0, $ $x \in \mathbb{R}^d.
$ It is easy to see that, for any $s \ge 0$, $y \in \mathbb{R}^d$,
$$
\partial_{s} v (s,y) + \mathrm{Tr} [
a \Big (\frac{s}{\lambda}, \frac{y}{\sqrt{\lambda}} \Big) D^2
v(s,y)] + \frac{1}{ \sqrt{\lambda}} b \Big ( \frac{s}{\lambda},
\frac{y}{\sqrt{\lambda}} \Big) \cdot Dv(s,y)
- v (s,y) =
\frac{1}{\lambda}{f \Big( \frac{s}{\lambda},
\frac{y}{\sqrt{\lambda}} \Big)}.
$$
Now  the  spatial H\"older seminorms of $ (s, y ) \mapsto
a(\frac{s}{\lambda}, \frac{y}{\sqrt{\lambda}})$ and $(s, y ) \mapsto
b( \frac{s}{\lambda}, \frac{y}{\sqrt{\lambda}}) $
are clearly independent on $\lambda \ge 1$ and
on $s \ge 0$.
By \cite[Theorem 2.4]{KP},
we deduce in particular, for any $\lambda \ge 1$,
$$
\sup_{s \ge 0} \|D v (s, \cdot) \|_{0}
 \le  \frac{C}{\lambda} \sup_{s \ge 0}\| f(s, \cdot) \|_{\theta },
$$ where $C$ is independent of $\lambda$. It follows
the assertion of  \cite[Lemma 6]{FGP}
since
$$
\sup_{t \ge 0} \|D u_{\lambda} (t, \cdot) \|_{0}
= \sqrt{\lambda}\sup_{s \ge 0} \|D v (s, \cdot) \|_{0}
 \le  \frac{C}{\sqrt{\lambda} } \sup_{s \ge 0}\| f(s, \cdot)
 \|_{\theta }.
$$
The proof of \cite[Theorem 5]{FGP} (which deals with the
stochastic flow)
 remains true  even with $\sigma$
 in \eqref{fy} by a straightforward modification.
}
\end{remark}

\begin{remark}\label{fine}{\em
An analogous of   Theorem \ref{th:flow1}
 holds   for \eqref{fy}
       requiring
  that  Hypotheses \ref{hy1}, \ref{hy2} and
\ref{hy3} are satisfied ``uniformly in time''.

One assumes  that $b$ and $\sigma_i$ are  continuous  functions
defined on  $[0,T] \times \RR^d$, $i=1, \ldots, k$. Moreover,
there  exists ${\theta}\in(0,1)$ such that
$b(t, \cdot )\in C_{}^{\theta }(\mathbb{R}^{d};\mathbb{R}^{d})$,
 $t \in [0,T]$, and $\sup_{t \in [0,T]}
 \|b(t, \cdot ) \|_{C_{}^{\theta }(\mathbb{R}^{d}, \RR^d)} < \infty$.
In addition, $\sigma_i(t, \cdot ) \in C^{3 }_b (\mathbb{R}^d,
 \mathbb{R}^d) $, $t \in [0,T]$,
 $$ \sup_{t \in [0,T]}
  \| \sigma_i (t, \cdot) \|_{C^{3 }_b (\mathbb{R}^d,
 \mathbb{R}^d) } < \infty,
 $$
  $i =1, \ldots, k$, and   one
 requires that condition \eqref{f7} holds.
 Theorem \ref{th:flow1} under these   assumptions
  may be established by adapting the (time-independent) proof
  given in
  the present paper. However, the complete argument,
even if it does not present
   special difficulties,
  is   considerably longer
 (for instance, one has  to prove the analogous of the
  Bismut-Elworthy-Li formula
 \eqref{bism} in the  time-dependent case).
}
\end{remark}

\smallskip

We close the section by an  application of the
stochastic flow.  We obtain
a Bismut-Elworthy-Li type formula for the derivative of
 the diffusion semigroup $(P_t)$ associated to
\eqref{SDE} (compare with \cite{B} and \cite{EL}).  It seems the
first time that such formula is given for  diffusion semigroups
associated to  SDEs with  coefficients which are not locally
Lipschitz.

\begin {theorem} \label{bismut} Let $f :\mathbb{R}^d \to \mathbb{R} $ be
uniformly continuous and bounded. For any $x, \, h \in
\mathbb{R}^d$, we have (cf. \eqref{bism})
$$
D_h P_t f (x) =
 \frac{1}{t} \mathbb{E} [ f(\phi_t(x))
\,  \int_0^t \langle (\sigma^{*}a^{-1})
(\phi_u(x)) D_h \phi_u (x) ,
dW_u \rangle
],\;\;\; t>0,\; x \in \mathbb{R}^d,
$$
where $\langle D P_t f(x), h \rangle = D_h P_t f (x)$ and  $D
\phi_u(x) $ solves \eqref{dove} with $s=0$
 (we set $\phi_u(x)= \phi_{0,u}(x)$).
\end{theorem}
\begin{proof}
We  prove the formula when $f \in C^{\infty}_b (\mathbb{R}^d)$. Indeed,
then,  by a straightforward uniform approximation of $f$, one can
obtain the formula in the general case.

Let $\vartheta :{\mathbb{ R}}^{d}\rightarrow {\mathbb{R}}$ be a
smooth test function such that $0\leq
\vartheta (x)\leq 1$, $x\in {\mathbb{R}}^{d}$,
$\vartheta (x)=\vartheta (-x)$,
 $\int_{{\mathbb{R}}^{d}}\vartheta (x)dx=1$,
 $\mathrm{supp}\,(\vartheta )\subset $ $B(0,2)$,
 $\vartheta (x)=1$ when
$x\in B(0,1)$. For any $n \ge 1$, let $\vartheta _{{n}}(x) ={n
}^{d}\vartheta (n x )$. Define $b_n = b * \vartheta _{{n}}$.

We have that  $b_{n}$ is  a  $C_{}^{\infty}$ and Lipschitz vector
field such that $b - b_n \in
C^{\theta}_b(\mathbb{R}^d;\mathbb{R}^d)$ and
$\| b - b_n\|_{C^{\theta}_b }$ tends to 0 as $n \to \infty$.
Let $(\phi^n_t)$ be the associated flow of smooth diffeomorphisms
which solves the SDE involving $b_n$
and let $(P_t^n)$ be
the corresponding  diffusion semigroup.
The  Bismut-Elworthy-Li formula for $(P_t^n)$ is given by
$$
D_h P_t^n f (x)= \frac{1}{t} \mathbb{E} [ f(\phi_t^n (x)) \,
 \int_0^t \langle (\sigma^{*}a^{-1}) (\phi_u^n(x)) D_h \phi^n_u (x) ,
dW_u \rangle ],\; t>0,\; x \in \mathbb{R}^d, \, n
 \in  \mathbb N.
$$
Note that
$
D_h P_t^n f (x)
 = \mathbb{E} [\langle D
 f (\phi_t^n (x)), D_h\phi_t^n (x) \rangle ]
$. Passing to the limit as $n \to \infty$, using the estimates
\eqref{stability1} and \eqref{stability2}, we get
$$
D_h P_t f (x)
 =  \mathbb{E} [ \langle D f (\phi_t (x)), D_h\phi_t (x) \rangle ]
= \frac{1}{t} \mathbb{E} [ f(\phi_t(x))
 \int_0^t \langle \sigma^{-1} (\phi_u(x)) D_h \phi_u (x) ,
dW_u \rangle ],
$$
for any  $ t>0,\; x \in \mathbb{R}^d .$
\end{proof}

\def\dopo{

\section{A flow in a random Holderian environment}

DA SISTEMARE ???

\label{sec:random} A simple problem which can be solved using the
theory developed here is that of a flow in a quenched Gaussian
vectorfield whose covariance is not smooth.

Consider a centered Gaussian vectorfield $V : \mathbb{R}^d \to
\mathbb{R}^d$ with covariance
\begin{equation}
\label{eq:cov} \mathbb{E}[V^a(x) V^b(y)] = \int_{\mathbb{R}^d}
\frac{\delta_{ab}|k|^2 -k^a k^b}{|k|^2} \frac{e^{i k\cdot (x-y)}
dk}{(1+|k|^2)^{\theta/2}} , \qquad a,b = 1,\dots,d
\end{equation}
where $d < \theta < d+1$. Due to the form of the covariance this
random vector-field does not allow smooth sample-paths. The
following lemma provides the existence of an Holderian (and
solenoidal) version of $V$:

\begin{lemma}
There exists a version of $V$ which belongs to
$C^\theta_{\text{loc}}$ for any $\theta < (\theta-d/2)$, linear
growth and such that
 $div V = 0$ in distributional sense.
\end{lemma}
\begin{proof}
Let us compute
$$
\mathbb{E}[|V(x)-V(y)|^2] = 2 (d-1) \int_{\mathbb{R}^d} \frac{1-e^{i
k\cdot (x-y)} dk}{(1+|k|^2)^{\theta/2}} = \Phi(|x-y|)
$$
and it is not difficult to estimate $\Phi(r) \le C$ and $\Phi(r)
\sim r^{\theta-d} $ for $r\to 0$ . Then by the Gaussian law we have
also
$$
\mathbb{E} [|V(x)-V(y)|^p] \le c_p  [\Phi(|x-y|)]^{p/2} \sim
|x-y|^{p(\theta-d)/2}
$$
for $|x-y| \le 1$ so that by the Kolmogorov lemma we are able to
find a version of $V$ which is locally $\theta$-H\"older continuous
for any exponent $\theta < (\theta-d)/2$. To estimate the pathwise
growth of $V$ note that $ \mathbb{E} \int |V(x)|^p (1+|x|)^{- \gamma
p} dx < \infty $ for any $2\le p < \infty$ and any $\gamma
d/p$. This  and the continuty of $V(x)$ imply that a.s.  the
bound $|V(x)| \le C(1+|x|)^{\gamma}$ holds for some random constant
$C$. Choosing $d/p < 1$ we have that $V$ has at most linear growth.

 To prove that $V$ has a solenoidal version consider
a smooth, compaclty supported test function $\theta$ , then
\begin{equation*}
\begin{split}
 & E \left(\int_{\mathbb{R}^d} \sum_{a=1}^d D_a \theta(x)
  V^a(x) dx\right)^2 =
\int_{\mathbb{R}^d} dx \int_{\mathbb{R}^d} dy \sum_{a,b=1}^d \nabla_a
\theta(x) \nabla_b \theta(y) E [V^a(x)V^b(y)]
\\ & \qquad  =
\int_{\mathbb{R}^d}
\sum_{a,b=1}^d  \frac{\delta_{ab}|k|^2 -k^a k^b}{|k|^2}
\frac{ dk}{(1+|k|^2)^{\theta/2}} k^a k^b |\widehat \theta(k)|^2 = 0
\end{split}
\end{equation*}
where $\widehat \theta$ is the Fourier transform of $\theta$. Then
$\langle D \theta, V\rangle = \langle \theta, div V \rangle= 0$
a.s..
\end{proof}


}

\end{document}